\documentclass[11pt]{amsart}

\usepackage[article]{robertpreamble}
\usepackage[margin=1in]{geometry}

\title[Volume Growth under Positive Intermediate Curvature]{Volume Growth under Positive Intermediate Curvature and a Ricci Lower Bound}
\author{Robert Koirala}
\address{Department of Mathematics, University of California San Diego}
\email{rkoirala@ucsd.edu}

\begin{document}

\begin{abstract}
    Let $(M^n,g)$ be a complete Riemannian manifold with $\Ric\ge-kg$ and uniformly positive $m$-intermediate curvature in the sense of Brendle--Hirsch--Johne. We prove that the Fisher eigenvalue $\lambda_{n-m+1}$ is small, after heat averaging, below the curvature scale $k^{-1}$. Consequently, balls have polynomial volume growth of order $R^{m-1}$ for $R\le k^{-1/2}$. At larger scales we obtain the corresponding estimate with an exponential factor $\exp(C\sqrt k R)$, and an example shows that this factor is necessary.
\end{abstract}

\maketitle

\section{Introduction}

A guiding principle in Riemannian geometry is that positively curved manifolds are geometrically small and topologically rigid. For example, the Bonnet--Myers theorem says that a complete manifold with a uniformly positive Ricci curvature has bounded diameter and finite fundamental group. Positivity of scalar curvature, however, no longer implies compactness. Nevertheless, several conjectures of Gromov predict that it still forces a codimension-two loss in large-scale geometry, expressed through volume growth, Urysohn width, and macroscopic dimension \cite{Gromov,GromovWidth,GromovProblems}.

Motivated by this circle of questions, we study volume growth under the intermediate-curvature condition introduced by Brendle--Hirsch--Johne \cite{BHJ}, which interpolates between Ricci and scalar curvature. For an orthonormal frame $e_1,\ldots,e_n$ and $1\le m\le n-1$, define the $m$-intermediate curvature by
\begin{equation}\label{eq:intermediate-curvature}
    \cC_m(e_1,\ldots,e_m)
    :=\sum_{p=1}^m\sum_{q=p+1}^n
    \Rm(e_p,e_q,e_p,e_q).
\end{equation}
Our convention is that the round sphere has positive sectional curvature. The endpoint identities
\[
    \cC_1(e_1)=\Ric(e_1,e_1),
    \qquad
    2\cC_{n-1}=\scal
\]
make this interpolation explicit. Geometric and topological aspects of nonnegative or positive intermediate curvature have been developed in \cite{BHJ,Chen,ChenHong,CJW,CJKL,Xu}.

Our main result shows that a uniform positive lower bound for $\cC_m$, together with a Ricci lower bound, forces a loss of $n-m+1$ dimensions in the volume growth below the scale determined by the negative part of the Ricci curvature.

\begin{theorem}\label{thm:volume}
    Let $(M^n,g)$ be complete and connected, and suppose that
    \begin{equation}\label{eq:curvature-assumptions}
        \Ric\ge-kg,
        \qquad
        \cC_m\ge\kappa>0.
    \end{equation}
    There are constants $C=C(n,m)<\infty$ and $C_n=C(n)<\infty$ such that, for every $x\in M$ and every $R>0$ with $kR^2\le1$,
    \begin{equation}\label{eq:volume-local}
        \vol B(x,R)
        \le C R^n(1+\kappa R^2)^{-(n-m+1)/2}.
    \end{equation}
    For arbitrary $R>0$,
    \begin{equation}\label{eq:volume-global}
        \vol B(x,R)
        \le C\kappa^{-(n-m+1)/2}R^{m-1}e^{C_n\sqrt k R}.
    \end{equation}
\end{theorem}

Under the stronger assumption of nonnegative sectional curvature, the $R^{n-2}$ volume growth follows from Petrunin's integral bound for scalar curvature \cite{Gromov,Petrunin}. Gromov asked whether nonnegative Ricci curvature and a uniform positive scalar-curvature lower bound suffice for volume growth of order $R^{n-2}$ \cite[p.~114]{Gromov}; see also \cite{Cai}. The three-dimensional case of Gromov's conjecture was proved by Munteanu--Wang \cite{MunteanuWang}; refinements and alternative arguments appear in \cite{ChodoshLiStryker,HuangLiu,Wang,WeiXuZhang,BoZhu}. A centered higher-dimensional estimate was obtained by Wang--Xie--Zhu--Zhu \cite[Theorem~1.9]{WangXieZhuZhu}. Gromov's volume conjecture was recently proved independently by Antonelli \cite{antonelli2026}, Ge \cite{Ge}, and Kong--Zhu \cite{KongZhu}. Antonelli also proved the result under positive intermediate curvature when $k=0$.

A related codimension-two conjecture of Gromov asks whether a uniform positive scalar-curvature lower bound controls the $(n-2)$-dimensional Urysohn width \cite{GromovWidth,GromovProblems,GuthUryson}. In dimension three, stronger waist and width estimates were proved by Liokumovich--Maximo \cite{LiokumovichMaximo} and Liokumovich--Wang \cite{LiokumovichWang}; the higher-dimensional conjecture remains open. Gromov also conjectured that the universal cover of a closed $n$-manifold with positive scalar curvature has macroscopic dimension at most $n-2$ \cite{GromovMacroscopic}; see \cite{BolotovDranishnikov,DranishnikovDuality} for partial results. Connections between ball-volume estimates, codimension-two Urysohn width, and macroscopic scalar curvature are further developed in \cite{AlpertBalitskiyGuth,GuthLarge,Sabourau}. Kumar--Sen showed that the macroscopic version of Gromov's Urysohn-width conjecture is false in dimensions at least four \cite{KumarSen}. Related volume and dimension-drop results under bi-Ricci or other partial curvature conditions appear in \cite{AntonelliXu,CucinottaMondino,ZhouZhu,XingyuZhu}.

\begin{remark}\label{rem:sharpness}
    If the lower bound for $\cC_m$ is omitted, then $\bH^n(-k/(n-1))$ satisfies $\Ric=-kg$ and has exponential volume growth. Thus a Ricci lower bound alone cannot force higher codimensional polynomial volume growth.
    
    Conversely, positive intermediate curvature alone does not give a scale-free polynomial bound. Suppose $m\ge3$ and consider
    \[
        M=\bS^{n-m+1}(a)\times\bH^{m-1}(-B),
    \]
    where the two factors have constant sectional curvatures $a^{-2}$ and $-B$, respectively. If $E$ is an $m$-plane, $F=E^\perp$, and $f_1,\ldots,f_{n-m}$ is an orthonormal basis of $F$, then
    \[
        \cC_m(E)=\frac{\scal}{2}
        -\sum_{\alpha<\beta}\Rm(f_\alpha,f_\beta,f_\alpha,f_\beta).
    \]
    Since every sectional curvature of the product is at most $a^{-2}$,
    \begin{align*}
        \cC_m(E) \ge \binom{n-m+1}{2}a^{-2}-\binom{m-1}{2}B
        -\binom{n-m}{2}a^{-2}=(n-m)a^{-2}-\binom{m-1}{2}B.
    \end{align*}
    For every $B>0$, choosing $a$ sufficiently small gives $\cC_m\ge\kappa$, while the volume growth contains the factor \(\exp\bigl((m-2)\sqrt B\,R\bigr).\) This proves that no polynomial bound can follow from $\cC_m\ge\kappa$ without a uniform Ricci lower bound. Finally, taking $B=k/(m-2)$ gives $\Ric\ge-kg$ and an exponential factor $\exp(\sqrt{m-2}\,\sqrt k R)$. Hence the dependence on the curvature scale in \eqref{eq:volume-local} and \eqref{eq:volume-global} is necessary.
\end{remark}

\subsection{Idea of the proof}

The first step is to use the Li--Yau heat-kernel estimate \cite[Corollary~3.1]{LiYau} to control the volume ratio at scale $R$ by the pointed Nash entropy $\ceN_x(R^2)$ (see \eqref{eq:Nash-definition}):
\begin{equation}\label{eq:intro-entropy-volume}
    \log\frac{\vol B(x,R)}{R^n}
    \le \ceN_x(R^2)+C_n(1+kR^2).
\end{equation}
To estimate the Nash entropy, we use its tensorial refinement, the Fisher metric $g_t^F$ (see \eqref{eq:Fisher-metric}). If $\lambda_1\le\cdots\le\lambda_n$ are the eigenvalues of $g^{-1}g_t^F$, then
\begin{align}\label{eq:intro-box-nash}
    -\square\ceN_x(t)
    =\frac{1}{2t}\sum_{i=1}^n(1-\lambda_i(x,t)),
\end{align}
where $\square=\partial_t-\Delta_x$, and $P_t$ denotes the heat semigroup (see \eqref{eq:E-Fisher-trace} as well as \cite{chowkoirala2026fisher}). The main analytic estimate in Theorem~\ref{thm:Fisher-decay} is that whenever $kt\le1$,
\begin{equation}\label{eq:intro-Fisher-decay}
    P_t\lambda_{n-m+1}(\cdot,t)(x)
    \le C\min\{1,(\kappa t)^{-1/m}\}.
\end{equation}
This smallness implies that the heat flow sees at most $m-1$ macroscopic directions and produces a loss of $n-m+1$ dimensions in the volume growth. More precisely, integrating \eqref{eq:intro-box-nash} in time and using \eqref{eq:intro-Fisher-decay}, we obtain
\begin{equation}\label{eq:intro-nash-bound}
    \ceN_x(t)\le-\frac{n-m+1}{2}\log(1+\kappa t)+C
    \qquad (kt\le1),
\end{equation}
which combines with \eqref{eq:intro-entropy-volume} to give \eqref{eq:volume-local}.

To prove \eqref{eq:intro-Fisher-decay}, we pass to the $m$th exterior power of the Fisher metric, whose trace dominates $\lambda_{n-m+1}^m$ (see \eqref{eq:sigma-controls-lambda}). Positive $m$-intermediate curvature enters through the weighted Hodge--Weitzenb\"ock formula and converts this exterior-power quantity into curvature terms controlled by the nonnegative term $Q_k$ associated to a curvature-corrected version of $-t^2\square\ceN_x(t)$ (see \eqref{eq:Qk-definition}). Proposition~\ref{prop:corrected-source} controls the heat space-time integral of $Q_k$; averaging over times comparable to $t$ and taking the $m$th root gives \eqref{eq:intro-Fisher-decay}.

Section~\ref{sec:heat-quantities} records the heat-semigroup estimates and studies the Fisher metric and Nash entropy. Section~\ref{sec:exterior-power} proves \eqref{eq:intro-Fisher-decay}. Finally, Section~\ref{sec:entropy-volume} establishes \eqref{eq:intro-nash-bound} and \eqref{eq:intro-entropy-volume}, and completes the proof of Theorem~\ref{thm:volume}.

\subsection*{Acknowledgment}

The author thanks Bennett Chow for constant support and encouragement. Part of this work was carried out while the author was visiting MIT. The author thanks MIT for its hospitality.

\subsection*{LLM usage disclosure}

The author acknowledges the use of AI as an interactive writing and checking aid during the preparation of this manuscript. The author directed this use, selected and substantially revised any generated text, and independently verified the mathematical statements and proofs. The author takes full responsibility for the final content of the paper.

\section{Heat-kernel preliminaries}\label{sec:heat-quantities}

We first introduce the heat-kernel quantities used throughout the proof and then collect the semigroup estimates that control them. Throughout this section, $\square=\partial_t-\Delta_x$ acts in the source space-time variables. Unless a target subscript is displayed, all tensors and derivatives are taken at the source point.

\subsection{Heat semigroup estimates}

Let $K(x,y,t)$ be the minimal heat kernel of $(M,g)$. For a fixed source point $x\in M$ and time $t>0$, write
\begin{equation}\label{eq:heat-measure-potential}
    d\nu_{x,t}(y)
    :=K(x,y,t)\,dg(y)
    =(4\pi t)^{-n/2}e^{-f_{x,t}(y)}\,dg(y).
\end{equation}
When the context is clear, we write $f=f_{x,t}$ and $d\nu=d\nu_{x,t}$. Let $P_t$ denote the heat semigroup, i.e.,
\[
    P_tu(x)=\int_MK(x,y,t)u(y)\,dg(y).
\]
The Ricci lower bound and Bishop--Gromov comparison imply $P_t1=1$; for the relevant volume-growth criterion, see \cite[Theorem~11.8]{Grigoryan}.

For $k\ge0$, set
\begin{equation}\label{eq:Lk}
    L_k(t)
    :=
    \begin{cases}
        \displaystyle\frac{2kt}{1-e^{-2kt}},&k>0,\\[1.2ex]
        1,&k=0.
    \end{cases}
\end{equation}
The function $L_k$ is increasing, $L_k(t)=1+O(kt)$ as $kt\downarrow0$, and $L_k(t)\sim2kt$ as $kt\to\infty$.

We collect the semigroup inequalities used below. The first three are the gradient, reverse Poincar\'e, and reverse logarithmic Sobolev inequalities associated with the curvature-dimension condition $\Ric\ge-kg$.

\begin{proposition}\label{prop:semigroup-estimates}
    For every bounded smooth function $u$ and every $t>0$,
    \begin{align}
        |\nabla P_tu|^2
        &\le e^{2kt}P_t|\nabla u|^2,
        \label{eq:BE-gradient}\\
        P_t(u^2)-(P_tu)^2
        &\ge \frac{2t}{L_k(t)}|\nabla P_tu|^2.
        \label{eq:reverse-Poincare}
    \end{align}
    If $u>0$, then
    \begin{equation}\label{eq:reverse-log-Sobolev}
        P_t(u\log u)-P_tu\log P_tu
        \ge \frac{t}{L_k(t)}
        \frac{|\nabla P_tu|^2}{P_tu}.
    \end{equation}
    Moreover, there is a dimensional constant $C_n$ such that, if $u\ge0$ is bounded and $r\le s\le2r$, then
    \begin{equation}\label{eq:same-point-Harnack}
        P_ru(x)\le C_ne^{C_nkr}P_su(x).
    \end{equation}
\end{proposition}

\begin{proof}
    The gradient estimate \eqref{eq:BE-gradient} is \cite[Corollary~3.3.19]{BGL}. The reverse inequalities \eqref{eq:reverse-Poincare} and \eqref{eq:reverse-log-Sobolev} are \cite[Theorems~4.7.2(iv) and 5.5.2(v)]{BGL}. For the logarithmic Sobolev form of the Li--Yau inequality, see also \cite{BakryLedoux}. Estimate \eqref{eq:same-point-Harnack} is a direct consequence of the Li--Yau differential Harnack inequality \cite[Theorem~2.2(i)]{LiYau}; see also \cite{LiXu}.
\end{proof}

We shall also use the following heat-flow comparison. It is the Duhamel inequality for the minimal heat kernel and follows by compact exhaustion and the parabolic comparison principle; see \cite[Theorem~7.13 and \S9.2]{Grigoryan}. 

\begin{lemma}\label{lem:heat-potential}
    Let $u$ and $F$ be smooth on $M\times(0,T]$, with $F\ge0$. Suppose that $u$ is bounded below, $\square u\ge F$, and $u(\cdot,t)\to0$ locally uniformly as $t\downarrow0$. Then
    \begin{equation}\label{eq:heat-potential-comparison}
        u(x,T)\ge\int_0^TP_{T-r}F(\cdot,r)(x)\,dr.
    \end{equation}
\end{lemma}

\subsection{Fisher metric}

We define the symmetric nonnegative definite \(2\)-tensor at the base point \(x\) as
\begin{equation}\label{eq:Fisher-metric}
    (g^F_t)_x
    :=2t\int_M d_x f\otimes d_x f\,d\nu_{x,t}=2t
    \int_M
    d_x\log K(x,y,t)\otimes d_x\log K(x,y,t)
    \,d\nu_{x,t}(y).
\end{equation}
Following the information-geometry literature, we call it the \textit{Fisher metric}; see Fisher \cite{Fisher1925}, Rao \cite{MR15748}, and Amari--Nagaoka \cite{MR1800071}. The analogous construction for the conjugate heat kernel of a Ricci flow is studied in \cite{chowkoirala2026fisher}; see also \cite{ISS} for the Fisher information metric of a fixed-manifold heat kernel. Related heat-kernel metric and embedding constructions appear in \cite{BerardBessonGallot,GigliMantegazza,Portegies}. Let
\[
    \ceJ_t:=g^{-1}g^F_t
\]
be the Fisher endomorphism and order its eigenvalues increasingly:
\begin{equation}\label{eq:Fisher-eigenvalues}
    0\le\lambda_1(x,t)\le\cdots\le\lambda_n(x,t).
\end{equation}
On Euclidean space one has $g^F_t=g$ for every $t>0$. On $\bS^2(a)\times\R^{n-2}$, separation of the heat kernel shows that, for each fixed $t>0$, the two eigenvalues associated with the sphere factor tend to zero as $a\downarrow0$, while the remaining $n-2$ eigenvalues equal one. Thus a small Fisher eigenvalue records a direction that has been averaged out by the heat flow at scale $\sqrt t$. Conversely, Fisher eigenvalues close to one are related to almost-Euclidean splitting; see \cite{chowkoirala2026fisher}.

The next proposition gives the two estimates for the Fisher metric that will be used throughout the paper. Analogous estimates and monotonicity for the conjugate heat kernel of a Ricci flow appear in \cite{chowkoirala2026fisher}.

\begin{proposition}\label{prop:Fisher-estimates}
    If $0<r<t$, then
    \begin{equation}\label{eq:data-processing}
        \frac1t g^F_t\le\frac1r g^F_r.
    \end{equation}
    Moreover, for every $t>0$,
    \begin{equation}\label{eq:Fisher-upper}
        0\le g^F_t\le L_k(t)g.
    \end{equation}
\end{proposition}

\begin{proof}
    Fix $x\in M$ and a unit vector $V\in T_xM$, and set
    \[
        S_V(y,t):=\nabla_V^x\log K(x,y,t)=-\nabla_V^xf_{x,t}(y).
    \]
    Thus
    \begin{equation}\label{eq:Fisher-score}
        \frac1{2t}g^F_t(V,V)
        =\int_M S_V(y,t)^2\,d\nu_{x,t}(y).
    \end{equation}
    We first prove \eqref{eq:Fisher-upper}. Differentiating the heat semigroup at the source point gives
    \begin{equation}\label{eq:score-duality}
        \nabla_VP_tu(x)=\int_MuS_V\,d\nu_{x,t}.
    \end{equation}
    Since $\int_MS_V\,d\nu_{x,t}=0$, the reverse Poincar\'e inequality \eqref{eq:reverse-Poincare} implies
    \[
        \left|\int_MuS_V\,d\nu_{x,t}\right|^2
        \le\frac{L_k(t)}{2t}
        \int_M\left(u-\int_Mu\,d\nu_{x,t}\right)^2d\nu_{x,t}.
    \]
    The preceding inequality shows that $u\mapsto\nabla_VP_tu(x)$ extends to a bounded linear functional on the mean-zero subspace of $L^2(d\nu_{x,t})$, with norm at most $\sqrt{L_k(t)/(2t)}$. By the Riesz representation theorem and \eqref{eq:score-duality}, its representing function is $S_V$. Hence
    \[
        \int_MS_V^2\,d\nu_{x,t}
        \le\frac{L_k(t)}{2t}.
    \]
    Multiplying by $2t$ and using \eqref{eq:Fisher-score} proves \eqref{eq:Fisher-upper}.
    
    We next prove \eqref{eq:data-processing}. Fix $0<r<t$. The Chapman--Kolmogorov identity and its source derivative give
    \begin{align}
        K(x,z,t)
        &=\int_MK(x,y,r)K(y,z,t-r)\,dg(y),\label{eq:first-chapman}\\
        K(x,z,t)S_V(z,t)
        &=\int_MK(x,y,r)S_V(y,r)K(y,z,t-r)\,dg(y). \label{eq:second-chapman}
    \end{align}
    On a noncompact manifold, the second equality is justified first on compact exhaustions; the $L^2$ bound for $S_V(\cdot,r)$ proved above allows passage to the minimal heat kernel. For fixed $z$, define
    \[
        d\pi_z(y)
        :=\frac{K(x,y,r)K(y,z,t-r)}{K(x,z,t)}\,dg(y).
    \]
    Equation \eqref{eq:first-chapman} says that $\int_Md\pi_z=1$. Dividing \eqref{eq:second-chapman} by $K(x,z,t)$ therefore gives
    \[
        S_V(z,t)=\int_MS_V(y,r)\,d\pi_z(y).
    \]
    Jensen's inequality gives
    \[
        S_V(z,t)^2K(x,z,t)
        \le\int_M S_V(y,r)^2K(x,y,r)K(y,z,t-r)\,dg(y).
    \]
    Integrating in $z$, using Fubini and $P_{t-r}1=1$, yields
    \[
        \int_MS_V(z,t)^2\,d\nu_{x,t}(z)
        \le\int_MS_V(y,r)^2\,d\nu_{x,r}(y).
    \]
    Together with \eqref{eq:Fisher-score}, this is \eqref{eq:data-processing}.
\end{proof}

\begin{lemma}\label{lem:source-fourth-moment}
    For every $x\in M$ and $t>0$,
    \begin{equation}\label{eq:source-fourth-moment}
        \int_M|d_xf|^4\,d\nu_{x,t}
        \le\frac{C_nL_k(t)^2}{t^2}.
    \end{equation}
\end{lemma}

\begin{proof}
    Fix a unit vector $V\in T_xM$ and write $S_V=S_V(\cdot,t)$. If $q$ is a probability density with respect to $\nu_{x,t}$, then \eqref{eq:reverse-log-Sobolev} and \eqref{eq:score-duality} give
    \[
        \left|\int_MS_Vq\,d\nu_{x,t}\right|
        \le
        \left(\frac{L_k(t)}t
        \operatorname{Ent}_{\nu_{x,t}}(q)\right)^{1/2}.
    \]
    The entropy variational formula therefore gives, for every $a\in\R$,
    \begin{align*}
        \log\int_Me^{aS_V}\,d\nu_{x,t}
        =\sup_q\left\{
        a\int_MS_Vq\,d\nu_{x,t}
        -\operatorname{Ent}_{\nu_{x,t}}(q)
        \right\}\le\sup_{\rho\ge0}
        \left\{|a|\sqrt{\rho L_k(t)/t}-\rho\right\}
        =\frac{a^2L_k(t)}{4t}.
    \end{align*}
    Thus each component of $d_xf$ is sub-Gaussian with variance bounded by $CL_k(t)/t$. Integrating the resulting Gaussian tail and summing over an orthonormal basis of $T_xM$ proves \eqref{eq:source-fourth-moment}.
\end{proof}

\subsection{Nash entropy}

The \textit{pointed Nash entropy} is
\begin{equation}\label{eq:Nash-definition}
    \ceN_x(t)
    :=\int_M f_{x,t}\,d\nu_{x,t}-\frac n2=-\int_MK(x,y,t)\log K(x,y,t)\,dg(y)
    -\frac n2\log(4\pi t)-\frac n2.
\end{equation}
See \cite{Colding,Nash1958,Ni,XuNash} for work on Nash entropy in the static setting and \cite{BamlerEntropy} for the Ricci-flow setting.

\begin{lemma}\label{lem:Nash-upper}
    For every $x\in M$ and $t>0$,
    \begin{equation}\label{eq:Nash-upper-basic}
        \ceN_x(t)
        \le\frac n2\int_0^t\frac{L_k(r)-1}{r}\,dr.
    \end{equation}
\end{lemma}

\begin{proof}
    Fix $z\in M$ and set
    \[
        I_z(s):=\int_M|\nabla_y\log K(z,y,s)|^2\,d\nu_{z,s}(y).
    \]
    The standard Bochner calculation in the target variable gives
    \[
        \frac d{ds}\int_M|\nabla_y\log K|^2\,d\nu_{z,s}
        =-2\int_M\left(
        |\nabla_y^2\log K|^2
        +\Ric(\nabla_y\log K,\nabla_y\log K)
        \right)d\nu_{z,s}.
    \]
    On a noncompact manifold this identity follows by inserting the smooth distance cutoffs of Greene--Wu \cite{GreeneWu}. The Li--Yau Gaussian and gradient bounds make the error terms tend to zero. We use the same cutoff construction in the proof of Theorem~\ref{thm:Fisher-decay}. Since $\Ric\ge-kg$,
    \[
        \frac d{ds}\int_M|\nabla_y\log K|^2\,d\nu_{z,s}
        \le-\frac2n\int_M(\Delta_y\log K)^2\,d\nu_{z,s}
        +2k\int_M|\nabla_y\log K|^2\,d\nu_{z,s}.
    \]
    The formula $P_t1=1$ and integration by parts give
    \[
        \int_M\Delta_y\log K\,d\nu_{z,s}
        =-\int_M|\nabla_y\log K|^2\,d\nu_{z,s}.
    \]
    Hence Jensen's inequality implies
    \[
        \frac d{ds}\int_M|\nabla_y\log K|^2\,d\nu_{z,s}
        \le
        -\frac2n\left(\int_M|\nabla_y\log K|^2\,d\nu_{z,s}\right)^2
        +2k\int_M|\nabla_y\log K|^2\,d\nu_{z,s}.
    \]
    The short-time heat-kernel asymptotics, localized by the Gaussian bound, show that the integral is asymptotic to $n/(2s)$ as $s\downarrow0$; see \cite{XuNash}. Scalar comparison for the preceding differential inequality therefore gives
    \begin{equation}\label{eq:target-Fisher}
        \int_M|\nabla_y\log K|^2\,d\nu_{z,s}
        \le\frac{nk}{1-e^{-2ks}}
        =\frac{nL_k(s)}{2s}
        \le C_n\left(\frac1s+k\right),
    \end{equation}
    with the usual limiting interpretation when $k=0$.
    
    Finally,
    \[
        \frac{d}{dt}\ceN_x(t)
        =\int_M|\nabla_y\log K(x,y,t)|^2\,d\nu_{x,t}(y)
        -\frac n{2t}.
    \]
    Since $\ceN_x(t)\to0$ as $t\downarrow0$, integration of \eqref{eq:target-Fisher} proves \eqref{eq:Nash-upper-basic}. The corresponding entropy calculation on complete noncompact manifolds is also discussed in \cite{Ni}.
\end{proof}

\subsection{Fisher metric and Nash entropy}\label{sec:corrected-source}

A direct calculation from \eqref{eq:Nash-definition}, using $\square(K\log K)=-K|d_x\log K|^2$, gives the relation between the Fisher metric and the Nash entropy:
\begin{equation}\label{eq:E-Fisher-trace}
    E(x,t):=\square\ceN_x(t)
    =\int_M|d_xf|^2\,d\nu_{x,t}-\frac n{2t}
    =\frac{\tr\ceJ_t-n}{2t}.
\end{equation}
To exploit this identity, we compute the evolution of $-t^2E$ and use Lemma~\ref{lem:heat-potential} to control $\ceN_x(t)$. However, the Bochner formula introduces a Ricci term, so the resulting $\Box E$ need not be nonnegative when $k>0$. We therefore add a scalar correction that compensates for the lower bound $\Ric\ge-kg$. This produces a nonnegative quantity $Q_k$ whose heat flow is uniformly controlled. Similar calculations for Ricci flow appear in \cite{chowkoirala2026fisher}.

Define
\begin{align}
    Q_k(x,t)
    :={}&4t^2\int_M\left(
    \left|\nabla_x^2f-\frac1{2t}g\right|^2
    +\Ric(\nabla_xf,\nabla_xf)
    \right)d\nu_{x,t}
    +2nktL_k(t),
    \label{eq:Qk-definition}\\
    \cE_k(x,t)
    :={}&-t^2E(x,t)
    +nk\int_0^t rL_k(r)\,dr.
    \label{eq:Ek-definition}
\end{align}

\begin{proposition}\label{prop:corrected-source}
    The quantities $Q_k$ and $\cE_k$ satisfy
    \begin{equation}\label{eq:corrected-source-identity}
        Q_k\ge0,
        \qquad
        \square\cE_k=\frac12Q_k,
        \qquad
        0\le\cE_k(x,t)
        \le C_nt L_k(t)^2.
    \end{equation}
    Consequently,
    \begin{equation}\label{eq:Qk-potential}
        \int_0^T
        P_{T-r}Q_k(\cdot,r)(x)\,dr
        \le2\cE_k(x,T)
        \le C_nT L_k(T)^2.
    \end{equation}
\end{proposition}

\begin{proof}
    For fixed $y$, the Bochner formula gives
    \begin{align*}
        \square\left[
        t^2K\left(\frac n{2t}-|d_xf|^2\right)
        \right]
        =2t^2K\left(
        \left|\nabla_x^2f-\frac1{2t}g\right|^2
        +\Ric(\nabla_xf,\nabla_xf)
        \right)+2t\operatorname{div}_x(K\nabla_xf).
    \end{align*}
    The last term integrates to zero in $y$, because
    \[
        \int_MK\nabla_xf\,dg(y)
        =-\int_M\nabla_xK\,dg(y)=0.
    \]
    Since
    \[
        -t^2E
        =t^2\left(\frac n{2t}
        -\int_M|d_xf|^2\,d\nu_{x,t}\right),
    \]
    integration in $y$ gives
    \[
        \square(-t^2E)
        =2t^2\int_M\left(
        \left|\nabla_x^2f-\frac1{2t}g\right|^2
        +\Ric(\nabla_xf,\nabla_xf)
        \right)d\nu_{x,t}.
    \]
    Since the time derivative of the correction term in \eqref{eq:Ek-definition} is $nktL_k(t)$, this proves $\square\cE_k=Q_k/2$.
    
    By \eqref{eq:Fisher-upper},
    \[
        \langle\Ric,g^F_t\rangle
        \ge-k\tr\ceJ_t
        \ge-nkL_k(t).
    \]
    Moreover,
    \[
        4t^2\int_M\Ric(\nabla_xf,\nabla_xf)\,d\nu_{x,t}
        =2t\langle\Ric,g^F_t\rangle.
    \]
    It follows directly from \eqref{eq:Qk-definition} that $Q_k\ge0$. Furthermore,
    \[
        -t^2E=\frac t2(n-\tr\ceJ_t)\le\frac{nt}{2},
    \]
    while monotonicity of $L_k$ and $2kt\le L_k(t)$ imply
    \[
        nk\int_0^t rL_k(r)\,dr
        \le\frac{nkt^2}{2}L_k(t)
        \le\frac{nt}{4}L_k(t)^2.
    \]
    This proves the upper bound for $\cE_k$.
    
    It remains to prove the lower bound and the heat flow estimate. On each finite time interval, \eqref{eq:Fisher-upper} gives
    \[
        \cE_k(x,t)
        \ge\frac{nt}{2}(1-L_k(t)).
    \]
    The right-hand side is bounded on compact time intervals and tends to zero as $t\downarrow0$. Lemma~\ref{lem:heat-potential}, applied with $F=Q_k/2$, therefore gives
    \[
        \cE_k(x,T)\ge\frac12\int_0^T
        P_{T-r}Q_k(\cdot,r)(x)\,dr\ge0.
    \]
    This proves both the lower bound in \eqref{eq:corrected-source-identity} and \eqref{eq:Qk-potential}.
    
    On a noncompact manifold, the differentiations above are justified on compact exhaustions. Local parabolic estimates and the heat-kernel gradient estimate of Kotschwar \cite{Kotschwar}, together with the Gaussian upper bound, give uniform integrable majorants on compact subsets of the source space-time variables, so the exhaustion limit is valid.
\end{proof}

\section{Exterior powers and Fisher eigenvalue decay}\label{sec:exterior-power}

In this section we prove the main analytic estimate for the Fisher eigenvalues.

\begin{theorem}\label{thm:Fisher-decay}
    Under the hypotheses of Theorem~\ref{thm:volume}, for every $x\in M$ and $t>0$, there are constants $C=C(n,m)<\infty$ and $C_n=C(n)<\infty$ such that
    \begin{equation}\label{eq:Fisher-main}
        P_t\bigl(\lambda_{n-m+1}(\cdot,t)\bigr)(x)
        \le
        \min\left\{
        L_k(t),
        C e^{C_nkt}L_k(2t)^{1+1/m}
        (\kappa t)^{-1/m}
        \right\}.
    \end{equation}
    In particular, if $kt\le1$, then
    \begin{equation}\label{eq:Fisher-local}
        P_t\bigl(\lambda_{n-m+1}(\cdot,t)\bigr)(x)
        \le C\min\{1,(\kappa t)^{-1/m}\}.
    \end{equation}
\end{theorem}

To prove \eqref{eq:intro-Fisher-decay}, we pass to the $m$th exterior power of the Fisher metric, whose trace dominates $\lambda_{n-m+1}^m$; see \eqref{eq:sigma-controls-lambda}. Positive $m$-intermediate curvature enters through the weighted Hodge--Weitzenb\"ock formula \eqref{eq:curvature-lower} and converts this exterior-power quantity into curvature terms controlled by $Q_k$; see Lemmas~\ref{lem:weighted-Weitzenbock} and~\ref{lem:Ricci-contraction}. After suppressing the $L_k$-factors, one obtains schematically
\[
    \kappa P_s(\lambda_{n-m+1}(\cdot,t)^m)
    \lesssim t^{-1}(P_sQ_k(\cdot,t)+1),
    \qquad s\simeq t;
\]
see \eqref{eq:weighted-eigenvalue}. Proposition~\ref{prop:corrected-source}, followed by averaging over comparable times using \eqref{eq:data-processing} and taking the $m$th root, gives the second term in \eqref{eq:Fisher-main} and hence \eqref{eq:Fisher-local}.

\subsection{Exterior powers}

We first isolate the exterior-power quantity that controls the desired eigenvalue. Set
\begin{equation}\label{eq:sigma-m}
    \sigma_m(\ceJ_t)
    :=\tr_{\Lambda^m}\left(\bigwedge^m\ceJ_t\right).
\end{equation}
In a $g$-orthonormal eigenbasis of $\ceJ_t$,
\begin{equation}\label{eq:sigma-controls-lambda}
    \sigma_m(\ceJ_t)
    =\sum_{|I|=m}\prod_{i\in I}\lambda_i
    \ge\lambda_{n-m+1}(x,t)^m.
\end{equation}
Thus $\sigma_m(\ceJ_t)$ measures the $m$-dimensional size of the Fisher metric. The inequality follows by retaining the term $I=\{n-m+1,\ldots,n\}$. To prove Theorem~\ref{thm:Fisher-decay}, it therefore suffices to control a heat average of $\sigma_m(\ceJ_t)$.

To connect this algebraic quantity with the heat kernel, define, for fixed $t>0$ and $y\in M$, the source one-form
\begin{equation}\label{eq:alpha}
    \alpha_y(x):=\sqrt{2t K(x,y,t)}\,d_xf(x,y,t).
\end{equation}
Then
\begin{equation}\label{eq:alpha-covariance}
    g^F_t=\int_M\alpha_y\otimes\alpha_y\,dg(y).
\end{equation}
For $\mathbf y=(y_1,\ldots,y_m)$, put
\begin{equation}\label{eq:Omega}
    \Omega_{\mathbf y}
    :=\alpha_{y_1}\wedge\cdots\wedge\alpha_{y_m}.
\end{equation}

Geometrically, $\Omega_{\mathbf y}$ is the oriented $m$-dimensional volume element generated by the one-forms $\alpha_{y_1},\ldots,\alpha_{y_m}$. The first identity below says that its averaged squared size is exactly $\bigwedge^m g_t^F$; after taking the trace, this is $\sigma_m(\ceJ_t)$. The second estimate controls the spatial variation of these volume elements by the nonnegative term $Q_k$ and a lower-order term.

\begin{lemma}\label{lem:wedge-energy}
    At every $(x,t)$,
    \begin{align}
        \frac1{m!}\int_{M^m}
        \Omega_{\mathbf y}\otimes\Omega_{\mathbf y}\,d\mathbf y
        &=\bigwedge^m g^F_t,
        \label{eq:Cauchy-Binet}\\
        \frac1{m!}\int_{M^m}|\nabla\Omega_{\mathbf y}|^2\,d\mathbf y
        &\le C\left[
        \frac{L_k(t)^{m-1}}t Q_k(x,t)
        +\frac{L_k(t)^{m+1}}t
        \right].
        \label{eq:wedge-gradient}
    \end{align}
\end{lemma}

\begin{proof}
    Equation \eqref{eq:Cauchy-Binet} is the Binet--Cauchy identity applied to \eqref{eq:alpha-covariance}. Differentiating \eqref{eq:alpha} gives
    \[
        \nabla\alpha_y
        =\sqrt{2t K}\left(
        \nabla_x^2f-\frac12d_xf\otimes d_xf
        \right).
    \]
    By \eqref{eq:source-fourth-moment} and \eqref{eq:Qk-definition},
    \begin{equation}\label{eq:alpha-gradient}
        \int_M|\nabla\alpha_y|^2\,dg(y)
        \le C_n\left(
        \frac{Q_k(x,t)}t
        +\frac{L_k(t)^2}t
        \right).
    \end{equation}
    Indeed, the Ricci lower bound and \eqref{eq:Fisher-upper} imply
    \begin{equation}\label{eq:Hessian-defect}
        Q_k(x,t)
        \ge4t^2\int_M
        \left|\nabla_x^2f-\frac1{2t}g\right|^2d\nu_{x,t}.
    \end{equation}
    Moreover,
    \[
        \nabla_x^2f-\frac12d_xf\otimes d_xf
        =\left(\nabla_x^2f-\frac1{2t}g\right)
        +\frac12\left(\frac1t g-d_xf\otimes d_xf\right).
    \]
    Since
    \[
        \int_M|\nabla\alpha_y|^2\,dg(y)
        =2t\int_M
        \left|\nabla_x^2f-\frac12d_xf\otimes d_xf\right|^2d\nu_{x,t},
    \]
    equation \eqref{eq:Hessian-defect} and the fourth-moment estimate \eqref{eq:source-fourth-moment} prove \eqref{eq:alpha-gradient}. Also,
    \begin{equation}\label{eq:alpha-L2}
        \int_M|\alpha_y|^2\,dg(y)
        =\tr\ceJ_t\le nL_k(t).
    \end{equation}
    Differentiating \eqref{eq:Omega} and applying Cauchy--Schwarz gives
    \[
        |\nabla\Omega_{\mathbf y}|^2
        \le m\sum_{j=1}^m|\nabla\alpha_{y_j}|^2
        \prod_{\ell\ne j}|\alpha_{y_\ell}|^2.
    \]
    Integrating the variables separately and using \eqref{eq:alpha-gradient} and \eqref{eq:alpha-L2} proves
    \eqref{eq:wedge-gradient}.
\end{proof}

\subsection{Weitzenb\"ock formula}

The curvature hypothesis in Theorem~\ref{thm:volume} enters through the Hodge--Weitzenb\"ock formula on $m$-forms (see \eqref{eq:curvature-lower}). Let $\cR_m$ denote its curvature endomorphism:
\begin{equation}\label{eq:Weitzenbock}
    \Delta_H=\nabla^*\nabla+\cR_m.
\end{equation}
If $e^I=e^{i_1}\wedge\cdots\wedge e^{i_m}$ is a simple unit form and $E_I=\operatorname{span}\{e_{i_1},\ldots,e_{i_m}\}$, then
\begin{equation}\label{eq:curvature-on-form}
    \langle\cR_me^I,e^I\rangle
    =\sum_{i\in I,\,j\notin I}\Rm(e_i,e_j,e_i,e_j).
\end{equation}
See \cite{GallotMeyer,Labbi,Petersen} for the curvature endomorphism on forms. Combining this with the definition of Ricci curvature gives
\begin{equation}\label{eq:curvature-algebra}
    \tr_{E_I}\Ric+\langle\cR_me^I,e^I\rangle
    =2\cC_m(E_I).
\end{equation}

Diagonalize $\ceJ_t$ and set $p_I=\prod_{i\in I}\lambda_i$. Multiplying \eqref{eq:curvature-algebra} by $p_I$ and summing over $I$ gives the fundamental inequality
\begin{align}\label{eq:curvature-lower}
    \tr_{\Lambda^m}\left(
    \cR_m\circ\bigwedge^m\ceJ_t
    \right)
    +\sum_{|I|=m}p_I\tr_{E_I}\Ric
    =2\sum_{|I|=m}p_I\cC_m(E_I)
    \ge2\kappa\sigma_m(\ceJ_t).
\end{align}
By \eqref{eq:Cauchy-Binet}, the first term in \eqref{eq:curvature-lower} is the averaged curvature action on the $m$-dimensional volume elements $\Omega_{\mathbf y}$:
\[
    \tr_{\Lambda^m}\left(
    \cR_m\circ\bigwedge^m\ceJ_t
    \right)
    =\frac1{m!}\int_{M^m}
    \langle\cR_m\Omega_{\mathbf y},\Omega_{\mathbf y}\rangle
    \,d\mathbf y.
\]

We use the following weighted form of the Weitzenb\"ock formula to control this term.

\begin{lemma}\label{lem:weighted-Weitzenbock}
    Let $w>0$ be smooth, let $\chi\in C_c^\infty(M)$, and let $\Omega$ be a smooth $m$-form. Then
    \begin{align}
        \int_Mw\chi^2\langle\cR_m\Omega,\Omega\rangle\,dg
        \le C_m\int_Mw\bigl(\chi^2|\nabla\Omega|^2
        +|d\chi|^2|\Omega|^2+\chi^2|d\log w|^2|\Omega|^2\bigr)\,dg.
        \label{eq:weighted-Weitzenbock}
    \end{align}
\end{lemma}

\begin{proof}
    For every compactly supported $m$-form $\eta$, integration of \eqref{eq:Weitzenbock} gives
    \[
        \int_M\langle\cR_m\eta,\eta\rangle\,dg
        =\int_M\bigl(|d\eta|^2+|\delta\eta|^2-|\nabla\eta|^2\bigr)\,dg
        \le C_m\int_M|\nabla\eta|^2\,dg.
    \]
    Apply this with $\eta=\chi\sqrt w\,\Omega$ and use the product rule.
\end{proof}

The second term in \eqref{eq:curvature-lower} is obtained by tracing the Ricci tensor over the $m$-plane $E_I$ spanned by the corresponding eigenvectors, with the weight $p_I=\prod_{i\in I}\lambda_i$. Equivalently, it is
\[
    \sum_{|I|=m}p_I\tr_{E_I}\Ric
    =\sum_{i=1}^n\Ric(e_i,e_i)\lambda_i
    \sigma_{m-1}(\lambda_1,\ldots,\widehat{\lambda_i},\ldots,\lambda_n).
\]
This term is controlled by the following estimate.

\begin{lemma}\label{lem:Ricci-contraction}
    We have
    \begin{equation}\label{eq:Ricci-contraction}
        \sum_{|I|=m}p_I\tr_{E_I}\Ric
        \le
        C\frac{L_k(t)^{m-1}}t Q_k(x,t).
    \end{equation}
\end{lemma}

\begin{proof}
    Write $r_i=\Ric(e_i,e_i)$. Then
    \[
        \sum_{|I|=m}p_I\tr_{E_I}\Ric
        =\sum_{i=1}^nr_i\lambda_i
        \sigma_{m-1}(\lambda_1,\ldots,\widehat{\lambda_i},\ldots,\lambda_n).
    \]
    By \eqref{eq:Fisher-upper},
    \[
        \sigma_{m-1}(\lambda_1,\ldots,\widehat{\lambda_i},\ldots,\lambda_n)
        \le\binom{n-1}{m-1}L_k(t)^{m-1}.
    \]
    Consequently,
    \[
        \sum_{|I|=m}p_I\tr_{E_I}\Ric
        \le C L_k(t)^{m-1}\sum_{r_i>0}r_i\lambda_i.
    \]
    Using $r_i\ge-k$, we obtain
    \begin{align*}
        \sum_{r_i>0}r_i\lambda_i
        \le\sum_{i=1}^nr_i\lambda_i+nkL_k(t)=\langle\Ric,g^F_t\rangle+nkL_k(t)
        \le\frac{Q_k(x,t)}{2t},
    \end{align*}
    where the last inequality follows from \eqref{eq:Qk-definition}. This proves the lemma.
\end{proof}

\begin{proof}[Proof of Theorem~\ref{thm:Fisher-decay}]
    We begin with a fixed-time estimate. For every $z\in M$ and $s,t>0$,
    \begin{align}
        P_s\bigl(\lambda_{n-m+1}(\cdot,t)^m\bigr)(z)
        \le\frac C\kappa\bigg[
        \frac{L_k(t)^{m-1}}t
        P_sQ_k(\cdot,t)(z)
        +\frac{L_k(t)^{m+1}}t
        +L_k(t)^m\left(\frac1s+k\right)
        \bigg].
        \label{eq:weighted-eigenvalue}
    \end{align}
    
    Let $w(x)=K(z,x,s)$. By the Greene--Wu exhaustion theorem \cite{GreeneWu}, there is a smooth proper function $\rho\ge0$ satisfying $|d\rho|\le2$. Choose $\chi\in C_c^\infty([0,\infty))$ with $0\le\chi\le1$ and $\chi=1$ on $[0,1]$, and set
    \[
        \chi_R(x)=\chi(\rho(x)/R).
    \]
    Then $|d\chi_R|\le C/R$ and $\chi_R\to1$ pointwise.
    
    For fixed $\mathbf y$, apply Lemma~\ref{lem:weighted-Weitzenbock} with $\Omega=\Omega_{\mathbf y}$ and $\chi=\chi_R$. The curvature is bounded on the compact support of $\chi_R$, so Fubini's theorem applies for fixed $R$. Integrate \eqref{eq:weighted-Weitzenbock} in $\mathbf y$, divide by $m!$, and then use \eqref{eq:wedge-gradient}, \eqref{eq:curvature-lower}, and \eqref{eq:Ricci-contraction}. This gives
    \begin{align*}
        2\kappa\int_Mw\chi_R^2\sigma_m(\ceJ_t)\,dg
        \le C\bigg[&
        \frac{L_k(t)^{m-1}}t
        \int_Mw\chi_R^2Q_k(\cdot,t)\,dg
        +\frac{L_k(t)^{m+1}}t\\
        &+\int_Mw|d\chi_R|^2\sigma_m(\ceJ_t)\,dg
        +\int_Mw\chi_R^2|d\log w|^2
        \sigma_m(\ceJ_t)\,dg
        \bigg].
    \end{align*}
    
    Since $P_s1=1$,
    \[
        \int_Mw|d\chi_R|^2\sigma_m(\ceJ_t)\,dg
        \le\frac C{R^2}L_k(t)^m\longrightarrow0.
    \]
    Furthermore, \eqref{eq:target-Fisher} gives
    \[
        \int_Mw|d\log w|^2\,dg
        \le C_n\left(\frac1s+k\right),
    \]
    while $\sigma_m(\ceJ_t)\le\binom nmL_k(t)^m$. If $P_sQ_k(\cdot,t)(z)=\infty$, the conclusion is automatic. Otherwise, dominated convergence applies to the remaining terms. Letting $R\to\infty$ and using \eqref{eq:sigma-controls-lambda} proves \eqref{eq:weighted-eigenvalue}.
    
    We now average \eqref{eq:weighted-eigenvalue} in time. Fix $x\in M$ and $t>0$. For $r\in[t/2,t]$, \eqref{eq:data-processing} gives
    \begin{equation}\label{eq:eigenvalue-data-processing}
        \lambda_{n-m+1}(\cdot,t)
        \le\frac t r\lambda_{n-m+1}(\cdot,r)
        \le2\lambda_{n-m+1}(\cdot,r).
    \end{equation}
    Set $s=2t-r\in[t,3t/2]$. By \eqref{eq:same-point-Harnack},
    \begin{equation}\label{eq:compare-heat-times}
        P_t\bigl(\lambda_{n-m+1}(\cdot,t)\bigr)(x)
        \le C_ne^{C_nkt}
        P_{2t-r}\bigl(\lambda_{n-m+1}(\cdot,r)\bigr)(x).
    \end{equation}
    Jensen's inequality and \eqref{eq:weighted-eigenvalue} imply
    \begin{align*}
        P_t\bigl(\lambda_{n-m+1}(\cdot,t)\bigr)(x)
        \le Ce^{C_nkt}\kappa^{-1/m}
        \bigg[&
        \frac{L_k(r)^{m-1}}r
        P_{2t-r}Q_k(\cdot,r)(x)
        +\frac{L_k(r)^{m+1}}r\\
        &+L_k(r)^m\left(\frac1{2t-r}+k\right)
        \bigg]^{1/m}.
    \end{align*}
    Since $r\in[t/2,t]$, the preceding inequality gives
    \begin{equation}\label{eq:pointwise-r}
        P_t\bigl(\lambda_{n-m+1}(\cdot,t)\bigr)(x)
        \le Ce^{C_nkt}\kappa^{-1/m}
        \left[
        \frac{L_k(2t)^{m-1}}t
        P_{2t-r}Q_k(\cdot,r)(x)
        +\frac{L_k(2t)^{m+1}}t
        \right]^{1/m}.
    \end{equation}
    Here we used $t^{-1}+k\le Ct^{-1}L_k(2t)$.
    
    By \eqref{eq:Qk-potential},
    \begin{equation}\label{eq:Qk-time-average}
        \int_{t/2}^t P_{2t-r}Q_k(\cdot,r)(x)\,dr
        \le\int_0^{2t}P_{2t-r}Q_k(\cdot,r)(x)\,dr
        \le C_nt L_k(2t)^2.
    \end{equation}
    Average \eqref{eq:pointwise-r} over $[t/2,t]$ and use the concavity of $u\mapsto u^{1/m}$. It follows that
    \[
        P_t\bigl(\lambda_{n-m+1}(\cdot,t)\bigr)(x)
        \le
        Ce^{C_nkt}L_k(2t)^{1+1/m}
        (\kappa t)^{-1/m}.
    \]
    The first term in the minimum in \eqref{eq:Fisher-main} follows from \eqref{eq:Fisher-upper}. This proves \eqref{eq:Fisher-main}. Equation \eqref{eq:Fisher-local} follows because $e^{C_nkt}$ and $L_k(2t)$ are uniformly bounded when $kt\le1$.
\end{proof}

\section{Entropy and volume}\label{sec:entropy-volume}

We now use Theorem~\ref{thm:Fisher-decay}, the evolution equation \eqref{eq:E-Fisher-trace}, and Lemma~\ref{lem:heat-potential} to prove an upper bound for the Nash entropy, from which the volume estimate follows.

\begin{proposition}\label{prop:Nash-entropy}
    Under the hypotheses of Theorem~\ref{thm:volume}, there is a constant $C=C(n,m)$ such that, for every $x\in M$ and $t>0$,
    \begin{equation}\label{eq:Nash-main}
        \ceN_x(t)
        \le
        -\frac{n-m+1}{2}
        \log\left(1+\kappa\min\{t,k^{-1}\}\right)
        +C(1+kt),
    \end{equation}
    where $k^{-1}=+\infty$ when $k=0$.
\end{proposition}

\begin{proof}
By \eqref{eq:E-Fisher-trace} and the ordering of the Fisher eigenvalues,
\begin{align}
    -E(x,t)
    =\frac1{2t}\sum_{i=1}^n(1-\lambda_i(x,t))
    \ge\frac1{2t}
    \left[
    (n-m+1)(1-\lambda_{n-m+1}(x,t))
    -(m-1)(L_k(t)-1)
    \right].
    \label{eq:E-lambda}
\end{align}
Define
\begin{equation}\label{eq:B-k}
    B_k(t)
    :=\frac n2\int_0^t\frac{L_k(r)-1}{r}\,dr.
\end{equation}
Since $\square(-\ceN)=-E$, \eqref{eq:E-lambda} gives
\begin{equation}\label{eq:entropy-positive-source}
    \square\bigl(-\ceN+B_k\bigr)
    \ge\frac{n-m+1}{2t}
    \bigl(L_k(t)-\lambda_{n-m+1}(\cdot,t)\bigr)
    \ge0.
\end{equation}
By \eqref{eq:Nash-upper-basic}, $-\ceN+B_k\ge0$. Applying Lemma~\ref{lem:heat-potential} to \eqref{eq:entropy-positive-source} yields
\begin{equation}\label{eq:entropy-potential}
    -\ceN_x(T)+B_k(T)
    \ge\frac{n-m+1}{2}\int_0^T
    \frac{L_k(r)-P_{T-r}\lambda_{n-m+1}(\cdot,r)(x)}r\,dr.
\end{equation}

Suppose $r\le T/2$ and $kr\le1$. By the semigroup property and \eqref{eq:Fisher-local},
\begin{equation}\label{eq:entropy-Fisher-comparison}
    P_{T-r}\lambda_{n-m+1}(\cdot,r)
    =P_{T-2r}\left(P_r\lambda_{n-m+1}(\cdot,r)\right)
    \le C\min\{1,(\kappa r)^{-1/m}\}.
\end{equation}
Since $\lambda_{n-m+1}(\cdot,r)\le L_k(r)$, the integrand in \eqref{eq:entropy-potential} is nonnegative. Let $C_0$ be the constant in \eqref{eq:entropy-Fisher-comparison}, set $A=(2C_0)^m$, and integrate over
\[
    A\kappa^{-1}\le r\le\min\{T/2,k^{-1}\}.
\]
When this interval is nonempty, \eqref{eq:entropy-Fisher-comparison} and $L_k\ge1$ give
\begin{align*}
    \int_{A/\kappa}^{\min\{T/2,k^{-1}\}}
    \frac{L_k(r)-P_{T-r}\lambda_{n-m+1}(\cdot,r)(x)}r\,dr
    &\ge \int_{A/\kappa}^{\min\{T/2,k^{-1}\}}
    \left(1-C_0(\kappa r)^{-1/m}\right)\frac{dr}{r}\\
    &\ge \log\left(\kappa\min\{T/2,k^{-1}\}\right)-C.
\end{align*}
It follows that
\begin{equation}\label{eq:entropy-log}
    \int_0^T
    \frac{L_k(r)-P_{T-r}\lambda_{n-m+1}(\cdot,r)(x)}r\,dr
    \ge
    \log\left(1+\kappa\min\{T,k^{-1}\}\right)-C.
\end{equation}
If the indicated interval is empty, the logarithm in \eqref{eq:entropy-log} is bounded by a constant, so the same conclusion holds after increasing $C$.

Finally, if $k>0$, the change of variables $u=kr$ in the definition
\eqref{eq:B-k} gives
\[
    B_k(T)
    =\frac n2\int_0^{kT}
    \frac{2u/(1-e^{-2u})-1}{u}\,du.
\]
The function
\[
    u\longmapsto\frac{2u/(1-e^{-2u})-1}{u}
\]
is bounded on $(0,\infty)$. When $k=0$, one has $B_k\equiv0$. Consequently,
\begin{equation}\label{eq:B-k-bound}
    B_k(T)\le C_nkT.
\end{equation}
Combining \eqref{eq:entropy-potential}, \eqref{eq:entropy-log}, and \eqref{eq:B-k-bound} proves \eqref{eq:Nash-main}.
\end{proof}

\begin{proof}[Proof of Theorem~\ref{thm:volume}]
    The Li--Yau Gaussian upper estimate under $\Ric\ge-kg$ \cite[Corollary~3.1]{LiYau} gives
    \begin{equation}\label{eq:Gaussian-upper}
        K(x,y,t)
        \le
        \frac{C_n}{\vol B(x,\sqrt t)}
        \exp\left(
        C_nkt-\frac{d(x,y)^2}{C_nt}
        \right).
    \end{equation}
    Taking $-\log$, integrating with respect to $d\nu_{x,t}$, and discarding the nonnegative distance term yields
    \begin{equation}\label{eq:entropy-volume}
        \log\frac{\vol B(x,\sqrt t)}{t^{n/2}}
        \le\ceN_x(t)+C_n(1+kt).
    \end{equation}
    If $kt\le1$, then \eqref{eq:Nash-main} and \eqref{eq:entropy-volume} imply
    \[
        \vol B(x,\sqrt t)
        \le Ct^{n/2}(1+\kappa t)^{-(n-m+1)/2}.
    \]
    Taking $t=R^2$ proves \eqref{eq:volume-local}.
    
    It remains to prove the global estimate. Suppose first that $k>0$ and set $\rho=k^{-1/2}$. By \eqref{eq:volume-local},
    \begin{equation}\label{eq:volume-at-curvature-scale}
        \vol B(x,\rho)
        \le C\kappa^{-(n-m+1)/2}\rho^{m-1}.
    \end{equation}
    For $R\ge\rho$, Bishop--Gromov comparison with the simply connected space form of sectional curvature $-k/(n-1)$ gives
    \begin{equation}\label{eq:Bishop-large-scale}
        \frac{\vol B(x,R)}{\vol B(x,\rho)}
        \le C_ne^{C_n\sqrt k R}.
    \end{equation}
    Since $\rho^{m-1}\le R^{m-1}$, equations \eqref{eq:volume-at-curvature-scale} and \eqref{eq:Bishop-large-scale} prove \eqref{eq:volume-global} for $R\ge\rho$. If $R\le\rho$, then \eqref{eq:volume-local} gives
    \[
        \vol B(x,R)
        \le C\kappa^{-(n-m+1)/2}R^{m-1},
    \]
    which is stronger. When $k=0$, the same conclusion follows directly from \eqref{eq:volume-local}. This completes the proof of Theorem~\ref{thm:volume}.
\end{proof}

\bibliographystyle{amsalpha}
\renewcommand{\bibliofont}{\fontsize{8}{9}\selectfont}
\bibliography{references}

\end{document}